\documentclass[10pt]{elsarticle}
\usepackage[cp1251]{inputenc}
\usepackage[english]{babel}
\usepackage{amsmath}
\usepackage{amssymb}
\usepackage{amsfonts}
\usepackage{amsthm}
\usepackage{mathtools}
\usepackage{dsfont}
\usepackage{rotating}
\usepackage{graphicx}
\usepackage{floatflt,epsfig}
\usepackage{lineno,hyperref}
\usepackage{enumerate}
\usepackage{colortbl}
\usepackage{array,tabularx,tabulary,booktabs}
\usepackage{longtable}
\usepackage{multirow}
\usepackage{wrapfig}
\usepackage{subcaption}
\usepackage{pdflscape}
\usepackage[table]{xcolor}
\newcolumntype{^}{>{\currentrowstyle}}

\journal{Arxiv}
\newtheorem{lemma}{Lemma}

\newtheorem{theorem}{Theorem}
\newtheorem{corollary}{Corollary}
\newtheorem{proposition}{Proposition}
\newtheorem{remark}{Remark}
\newtheorem{problem}{Problem}
\begin{document}
\renewcommand{\abstractname}{Abstract}
\renewcommand{\refname}{References}
\renewcommand{\tablename}{Table}
\renewcommand{\arraystretch}{0.9}
\thispagestyle{empty}
\sloppy

\begin{frontmatter}
\title{A general construction of strictly Neumaier graphs and a related switching}
\author[01]{Rhys J. Evans}
\ead{rhysjevans00@gmail.com}

\author[02]{Sergey Goryainov}
\ead{sergey.goryainov3@gmail.com}

\author[01,03,04]{Elena~V.~Konstantinova}
\ead{e\_konsta@math.nsc.ru}

\author[01,03]{Alexander~D.~Mednykh}
\ead{smedn@mail.ru}

\address[01]{Sobolev Institute of Mathematics, Ak. Koptyug av. 4, Novosibirsk, 630090, Russia}
\address[02] {School of Mathematical Sciences, Hebei International Joint Research Center for Mathematics and Interdisciplinary Science, \\Hebei Normal University, Shijiazhuang  050024, P.R. China}
\address[03]{Novosibirsk State University, Pirogova str. 2, Novosibirsk, 630090, Russia}
\address[04]{Three Gorges Mathematical Research Center, China Three Gorges University, 8 University Avenue, Yichang 443002, Hubei Province, China}

%\tnotetext[grant]{The reported study was funded by ...}

\begin{abstract}
We present a construction of Neumaier graphs with nexus 1, which generalises two known constructions of Neumaier graphs. We also use W. Wang, L. Qiu, and Y. Hu switching to show that we construct cospectral Neumaier graphs. Finally, we show that several small strictly Neumaier graphs can be obtained from our construction, and give a geometric or algebraic description for each of these graphs. 
\end{abstract}

\begin{keyword}
edge-regular graph; regular clique; Neumaier graph; WQH-switching; perfect code 
\vspace{\baselineskip}
\MSC[2010] 05C25 \sep 05B45 \sep 05C69
\end{keyword}
% 05C25 graphs and groups
% 05E10 graphs and matrices
% 05E15 combinatorial problems concerning the classical groups;
\end{frontmatter}
%%%%%%%%%%%%%%%%%%%%%%%%%%%%%%%%%%%%%%%%%%%%%%%%%%%%%%%%%%%%%%%%%%%%%%%%%%%%%%%%%%%%%%%%%%%%%%%%%%%%%%%%%%%%%%%%%%%%%%%%%%%%%%%%%%

\section{Introduction}\label{Intro}
A regular clique in a finite regular graph is a clique such that every vertex that does not belong to the clique is adjacent to the same positive number of vertices in the clique. A regular clique in a graph can be equivalently viewed as a clique which is a part of an
equitable 2-partition (see \cite{BH12, GR01}), or a clique which is a completely regular code of radius 1 (see \cite{N92} and
\cite[p. 345]{BCN89}) in this graph.

It is well known that a clique in a strongly regular graph is regular if and
only if it is a Delsarte clique (see \cite{BHK07,BCN89}).
In \cite{N81}, A. Neumaier posed the problem of whether there exists a non-complete, edge-regular,
non-strongly regular graph containing a regular clique. A \emph{Neumaier graph}
is a non-complete edge-regular graph containing a regular clique and a \emph{strictly
Neumaier graph} is a non-strongly regular Neumaier graph. (These definitions are analogous to the definitions of Deza graphs and strictly Deza graphs \cite{EFHHH99}.)

Two families of strictly Neumaier graphs with $1$-regular cliques were found in \cite{GK18} and \cite{GK19}. In \cite{EGP19}, two strictly Neumaier graphs with $2^i$-regular cliques for every positive integer $i$ were constructed. In \cite{ADDK21}, strictly Neumaier graphs with few eigenvalues were studied.  

In this paper we present a generalisation of the constructions found in \cite{GK18} and \cite{GK19}; the general construction requires the existence of edge-regular graphs admitting a partition into perfect 1-codes of a certain size. In Section \ref{sec:Prelims} we give several definitions related to strictly Neumaier graphs, and the definition of WQH-switching. In Section \ref{sec:GeneralConstruction} we present the main results of this paper. A general construction of Neumaier graphs containing 1-regular cliques is given in Theorem \ref{GeneralConstruction}, followed by a criterion for the resulting graph to be a strictly Neumaier graph in Corollary \ref{cor:t2strict}. Further, we use a WQH-switching to show that our construction creates cospectral Neumaier graphs in Proposition \ref{Switching}.

In Section \ref{sec:Examples}, we present several examples of small strictly Neumaier graphs, and show that each of these graphs can be found using our general construction. For each of these graphs, we give a geometric or algebraic description of their structure. Finally, in Section \ref{sec:FurtherProblems} we present families of infinite edge-regular graphs, and ask if we can use these graphs to construct infinite families of strictly Neumaier graphs. In finding such families, we would be extending the constructions of certain graphs found in Section \ref{sec:Examples} to infinite families of graphs.

We would like to note that the main result of this paper, Theorem \ref{GeneralConstruction}, has already been referenced, presented, adjusted and used in \cite{ACDKZ23}. Furthermore, the strictly Neumaier graph we construct in Example \ref{graph65} was found independently by the authors of \cite{ACDKZ23}, who present an infinite family of strictly Neumaier graphs containing this graph as its smallest graph. In \cite{ACDKZ23}, there are also some interesting non-existence results for Neumaier graphs.

\section{Preliminaries}\label{sec:Prelims}

In this paper we only consider undirected graphs that contain no loops or multiple
edges. Let $\Gamma$ be such a graph. We denote by $V(\Gamma)$ the vertex set of $\Gamma$ and $E(\Gamma)$ the edge set of $\Gamma$. For a vertex $u \in V(\Gamma)$ we define the \emph{neighbourhood} of $u$ in $\Gamma$ to be the
set $\Gamma(u) = \{w \in V(\Gamma) : uw \in E(\Gamma)\}$.

Let $\Gamma$ be a graph and $v = |V(\Gamma)|$. The graph $\Gamma$ is called \emph{$k$-regular} if every vertex has
neighbourhood of size $k$. The graph $\Gamma$ is \emph{edge-regular} if it is non-empty, $k$-regular, and every pair of
adjacent vertices have exactly $\lambda$ common neighbours. Then $\Gamma$ is said to be edge-regular
with \emph{parameters} $(v, k, \lambda)$. The graph $\Gamma$ is \emph{co-edge-regular} if it is non-complete, $k$-regular and every pair of distinct non-adjacent vertices have exactly $\mu$ common neighbours. Then $\Gamma$ is said to be co-edge-regular with \emph{parameters} $(v, k, \mu)$.
The graph $\Gamma$ is \emph{strongly regular} if it is both edge-regular and co-edge-regular. If $\Gamma$ is
edge-regular with parameters $(v, k, \lambda)$, and co-edge-regular with parameters $(v, k, \mu)$, the
graph is called \emph{strongly regular} with \emph{parameters} $(v, k, \lambda, \mu)$.

A \emph{clique} in a graph $\Gamma$ is a set of pairwise adjacent vertices of $\Gamma$, and a clique of size $s$
is called an \emph{$s$-clique}. A clique $S$ in a regular graph $\Gamma$ is \emph{regular} if every vertex that does not belong to $S$ is
adjacent to the same number $m > 0$ of vertices in $S$. We then say that $S$ has \emph{nexus}
$m$ and is \emph{$m$-regular}. A set of cliques of a graph $\Gamma$ that partition the vertex set of $\Gamma$ is called a \emph{spread}
in $\Gamma$.

A \emph{Neumaier graph} with \emph{parameters} $(v, k, \lambda; m, s)$ is a non-complete edge-regular graph with parameters $(v,k,\lambda)$  containing an $m$-regular $s$-clique.  A \emph{strictly
Neumaier graph} is a Neumaier graph that is not strongly regular.

\subsection{Perfect codes}

Let $\Gamma$ be a simple undirected graph and $e \ge 1$ an integer. The \emph{ball} with radius $e$
and centre $u \in V(\Gamma)$ is the set of vertices of $\Gamma$ with distance at most $e$ to $u$ in $\Gamma$. A subset $C$ of $V(\Gamma)$ is called a \emph{perfect $e$-code} in $\Gamma$ if the balls with radius $e$ and centres in $C$ form a partition
of $V(\Gamma)$ (see \cite{B73, K86}). In particular, a perfect 1-code  is a subset of vertices $C$ such that every vertex not in $C$ is adjacent to a unique element of $C$.

\subsection{Wang-Qiu-Hu switching}

The \emph{spectrum} of a graph $\Gamma$ is the multiset of eigenvalues of the adjacency matrix of $\Gamma$ (see \cite{BH12}). Two graphs are \emph{cospectral} if they have the same spectra.

The following switching, which produces cospectral graphs, was discovered in \cite{WQH19} and applied in \cite{IM19} to obtain new strongly regular graphs.

\begin{lemma}[WQH-switching]\label{WQHswitching}
Let $\Gamma$ be a graph whose vertex set is partitioned as $C_1 \cup C_2 \cup D$. Assume that $|C_1| =|C_2|$ and that the induced subgraphs on $C_1, C_2$, and $C_1 \cup C_2$ are regular, where the degrees in the induced subgraphs on $C_1$ and $C_2$ are the same. Suppose that all $x\in D$ satisfy one of the following:
\begin{enumerate}
	\item $|\Gamma(x) \cap C_1| = |\Gamma(x) \cap C_2|$, or
    \item $\Gamma(x) \cap (C_1 \cup C_2)| \in \{C_1,C_2\}$.
\end{enumerate}
Construct a graph $\Gamma'$ from $\Gamma$ by modifying the edges between $C_1 \cup C_2$ and $D$ as follows:
\begin{equation*}
\Gamma'(x)\cap(C_1 \cup C_2) =
\begin{cases}
C_1, & \text{if } \Gamma(x)\cap(C_1 \cup C_2) = C_2;\\
C_2, & \text{if } \Gamma(x)\cap(C_1 \cup C_2) = C_1;\\
\Gamma(x)\cap(C_1 \cup C_2), & \text{otherwise},
\end{cases}
\end{equation*}
for all $x \in D$. Then $\Gamma'$ is cospectral with $\Gamma$.
\end{lemma}

\section{A general construction of strictly Neumaier graphs with 1-regular cliques}\label{sec:GeneralConstruction}

In this section we present a construction of Neumaier graphs with 1-regular cliques, which generalises the constructions found in \cite{GK18} and \cite{GK19}. We note that our construction was first introduced in the PhD thesis of R. J. Evans, the first author of this paper (see \cite[Theorem 5.1]{E20}).

%p. 4, l. 13. The parameter t is already introduced as the number of edge-regular graphs Γ(i). So, you cannot ‘define’ it again. It is more a condition that t = (λ+2)/a. 

%CHANGE - Throughout Section 3 - DONE

Let $\Gamma^{(1)},\dots,\Gamma^{(t)}$ be edge-regular graphs with parameters $(v,k,\lambda)$, and such that each $\Gamma^{(i)}$ has a partition of its vertices into perfect $1$-codes of size $a$, where $a=(\lambda+2)/t$.

For any $\ell\in \{1,\ldots, t\}$, let $H_{1}^{(\ell)},\dots,H_{v/a}^{(\ell)}$ denote the perfect $1$-codes that partition the vertex set of $\Gamma^{(\ell)}$.

%CHECK - Omega used instead of {1,...,v/a}

If $t\geq 2$, we also take  a $(t-1)$-tuple of permutations from
$\text{Sym}(\{1, \ldots, v/a \})$, denoted by $\Pi =(\pi_{2},\dots,\pi_{t})$. 

Using these graphs and the permutation $\Pi$, we define the graph $F_\Pi(\Gamma^{(1)},\ldots,\Gamma^{(t)})$ as follows.
\begin{enumerate}
  \item Take the disjoint union of the graphs $\Gamma^{(1)},\ldots,\Gamma^{(t)}$.
  \item For any $i \in \{1,\ldots,v/a\}$, add an edge between any two distinct vertices from $H^{(1)}_i \cup H^{(2)}_{\pi_2(i)} \cup \ldots \cup H^{(t)}_{\pi_t(i)}$ (which forms a $1$-regular clique of size $at$).
\end{enumerate}
Note that for $t=1$, our definition of $F_\Pi(\Gamma^{(1)},\ldots,\Gamma^{(t)})=F_\Pi(\Gamma^{(1)})$ does not use $\Pi$, so we do not need to choose $\Pi$ to construct our graph.

Now we show that the graph $F_\Pi(\Gamma^{(1)},\ldots,\Gamma^{(t)})$ is a Neumaier graph, and express the parameters of this graph in terms of the parameters of the original graphs, $t$ and $a$. 

\begin{theorem}\label{GeneralConstruction}
Let $\Gamma^{(1)},\dots,\Gamma^{(t)}$ be edge-regular graphs with parameters $(v,k,\lambda)$ and let $\Pi =(\pi_{2},\dots,\pi_{t})$ be a $(t-1)$-tuple of permutations from
{\rm Sym}$(\{1, \ldots, v/a \})$ when $t\geq 2$.

Further, suppose that each graph $\Gamma^{(\ell)}$ has a partition of its vertices into perfect $1$-codes $H_{1}^{(\ell)},\dots,H_{v/a}^{(\ell)}$, each of size $a$, where $a=(\lambda+2)/t.$

Then
\begin{enumerate}
	\item $F_{\Pi}(\Gamma^{(1)},\dots,\Gamma^{(t)})$ has (a spread of) $1$-regular cliques, each of size $\lambda+2$;
	\item $F_{\Pi}(\Gamma^{(1)},\dots,\Gamma^{(t)})$ is an edge-regular graph with parameters $(vt,k+\lambda+1,\lambda).$
\end{enumerate}
\end{theorem}
\begin{proof}
1. For $i\in \{1,\ldots, t\}$, consider the union of the sets $H_{i}^{(1)},H_{\pi_{2}(i)}^{(2)},\dots,H_{\pi_{t}(i)}^{(t)}$. As $H_{i}^{(j)}$ is a perfect $1$-code in $\Gamma^{(i)}$, every vertex in $V(\Gamma^{(i)})\setminus H_{i}^{(j)}$ is adjacent to a unique element of $H_{i}^{(j)}$. Therefore, each vertex of $F_{\Pi}(\Gamma^{(1)},\dots,\Gamma^{(t)})$ is adjacent to a unique vertex of $H_{i}^{(1)}\cup H_{\pi_{2}(i)}^{(2)}\cup \dots\cup H_{\pi_{t}(i)}^{(t)}$. After adding all possible edges in $H_{i}^{(1)}\cup H_{\pi_{2}(i)}^{(2)}\cup\dots\cup H_{\pi_{t}(i)}^{(t)}$, we see that this set is a $1$-regular clique.
	
2. The graph $F_{\Pi}(\Gamma^{(1)},\dots,\Gamma^{(t)})$ has $vt$ vertices by definition. 

Let $w$ be a vertex in $F_{\Pi}(\Gamma^{(1)},\dots,\Gamma^{(t)})$, where $w\in H_{i}^{(j)}$ for some $i,j$. Note that $w$ is adjacent to $k$ vertices of $\Gamma^{(j)}$ before adding the edges of the construction. Then $w$ gains $a-1$ new neighbours from $H_{i}^{(j)}$, and $a$ new neighbours from $H_{\pi_{m}(\pi^{-1}_{j}(i))}^{(m)}$, for each $m\not=j$ (where we take $\pi_1$ as the identity element in $\text{Sym}(\{1, \ldots, v/a \})$). Therefore, the degree of $w$ is $k+a-1+a(t-1)=k+\lambda+1$. 

%p. 5, l. 9. You assume here that π1 is the identity element of the symmetric group, for the case m = 1. Mention this, please. 

%CHANGE - Theorem 1 proof Paragraph 2 - DONE

%p. 5, l. 13. “... the same clique H(1)i ∪ H(2)π2(i) ∪ · · · ∪ H(t)πt(i)...” 

%CHANGE - Theorem 1 proof Paragraph 3 - DONE

Now let $u,w$ be adjacent vertices in $F_{\Pi}(\Gamma^{(1)},\dots,\Gamma^{(t)})$.  First we suppose $u,w$ lie in the same clique $H_{i}^{(1)} \cup H_{\pi_{2}(i)}^{(2)},\dots \cup H_{\pi_{t}(i)}^{(t)}$, for some $i$. As this clique is $1$-regular, $u,w$ must have $at-2=\lambda$ common neighbours. Otherwise by definition of $F_{\Pi}(\Gamma^{(1)},\dots,\Gamma^{(t)})$,  there must be integers $m,n$ and $j$ such that $m\not=n$ and $u\in H_{m}^{(j)},w\in H_{n}^{(j)}$. As $H_{m}^{(j)},H_{n}^{(j)}$ are disjoint perfect codes in $\Gamma^{(j)}$, there are no new common neighbours of $u,w$ coming from adding the edges in the construction of the graph. Therefore, all adjacent vertices have $\lambda$ common neighbours.
\end{proof}

%p. 2, l. 15. You mention that Theorem 1 generalises [12] and [13]. Maybe it would be useful that in the discussion after Theorem 1 you get back to this, and explain (shortly) how it generalises these constructions. Which specific cases correspond to [12] and [13]? 

%CHANGE - Remark 1 - DONE

%p. 5, Remark 1. This is an important result. Maybe state it as a theorem. I mean, it is good to know that all Neumaier graphs with t = 1 and a spread of regular cliques arise from this construction.  

% SGcomment: I completely agree with the referee that this should be a theorem. Please make an appropriate accent.

%CHANGE - Theorem 2 + discussion  - DONE

In \cite{S15}, the converse of this theorem is proven as a special case of a general result. We present this case in detail in the following theorem.

\begin{theorem}
    Let $\Gamma\in \text{NG}(v,k,\lambda;1,s)$ have a spread of $1$-regular cliques. Further, let $\Gamma^{\circ}$ be the graph created by removing all the edges in $\Gamma$ that lie in a clique of the spread, and let $t$ be the number of connected components of $\Gamma^{\circ}$. Then
    \begin{enumerate}
        \item $\Gamma^{\circ}$ is edge-regular with parameters $(v,k-\lambda-1,\lambda);$
        \item Each connected component has the same number of vertices, and can be partitioned into perfect $1$-codes of size $a=(\lambda+2)/t$.
    \end{enumerate} 
\end{theorem}
\begin{proof}
     As each clique in a spread is 1-regular, distinct vertices in a clique of $\Gamma$ cannot have a shared neighbour outside of the clique. Therefore, $s=\lambda+2$ and the result is a simple corollary of L. Soicher \cite[Theorem 6.1]{S15}.
\end{proof}
\begin{remark}
    The previous two results show that all Neumaier graphs with a spread of 1-regular cliques come from our construction. For example, the construction found in \cite{GK19} uses isomorphic distance-regular graphs as the underlying edge-regular graphs, and antipodal classes in the graphs are the perfect 1-codes found in our generalisation, for $t\ge 2$. The construction found in \cite{GK18} uses Cayley graphs with particular choices of generating sets that force the graphs to be Neumaier graphs. All of these generating sets contain all non-identity members of a subgroup, the cosets of which define a partition into perfect 1-codes in the graphs obtained by removing these elements from the generating set. These Cayley graphs are examples of our construction with $t=1$.
\end{remark}
% \begin{remark}
% We note that a converse to this theorem is also true. Let $\Gamma\in \text{NG}(v,k,\lambda;1,s)$ have a spread of $1$-regular cliques. The graph $\Gamma^{\circ}$ created by removing edges from the cliques of this spread in $\Gamma$ is also edge-regular by Soicher \cite[Theorem 6.1]{S15}. It also follows that connected components of $\Gamma^{\circ}$ can be partitioned into perfect $1$-codes, and the parameters have the same restrictions as the conditions of the statement of Theorem \ref{GeneralConstruction}.
% \end{remark}

In Section \ref{sec:Examples}, we show that the case for which $t=1$ can occur, and the construction can result in a strictly Neumaier graph. However, this is not necessarily true in all cases when $t=1$. The following corollary shows that for $t\geq 2$, our construction always results in a strictly Neumaier graph.

\begin{corollary}\label{cor:t2strict}
Let $\Gamma^{(1)},\dots,\Gamma^{(t)}$ be non-complete edge-regular graphs with parameters $(v,k,\lambda)$ and let $\Pi =(\pi_{2},\dots,\pi_{t})$ be a $(t-1)$-tuple of permutations from
{\rm Sym}$(\{1, \ldots, v/a \})$.

Further, suppose that each graph $\Gamma^{(\ell)}$ has a partition of its vertices into perfect $1$-codes $H_{1}^{(\ell)},\dots,H_{v/a}^{(\ell)}$, each of size $a$, where $a=(\lambda+2)/t.$

If $t\geq 2$, then 	$F_\Pi(\Gamma^{(1)},\ldots,\Gamma^{(t)})$ is a strictly Neumaier graph.
\end{corollary}
\begin{proof}
By Theorem \ref{GeneralConstruction}, we know that $F_\Pi(\Gamma^{(1)},\ldots,\Gamma^{(t)})$ is a Neumaier graph. Therefore, we need to prove that  $F_\Pi(\Gamma^{(1)},\ldots,\Gamma^{(t)})$ is not a co-edge-regular graph. 

Consider the subgraphs $\Gamma^{(1)}$ and $\Gamma^{(2)}$ of $F_\Pi(\Gamma^{(1)},\ldots,\Gamma^{(t)})$. Let $x \in H^{(1)}_{i_1}$ and $y \in H^{(2)}_{i_2}$, where $i_2 \ne \pi_2(i_1)$. Then $x$ and $y$ have exactly two common neighbours in $F_\Pi(\Gamma^{(1)},\ldots,\Gamma^{(t)})$: one common neighbour in $H^{(2)}_{\pi_2(i_1)}$ and one common neighbour in $H^{(1)}_{i_3}$, where $\pi_2(i_3) = i_2$.

%P5 L50 "Now consider the subgraphs $\Gamma^{(1)}$ and $\Gamma^{(2)}$."
%Comment: Here $\Gamma^{(1)}$ is enough. 

%CHANGE - Corollary 1 proof Paragraph 3 - DONE

%p. 5, l. 51. ‘Distinct’ seems redundant here. They are at distance 2. 

%CHANGE - Corollary 1 proof Pargraph 3 - DONE

Now consider the subgraph $\Gamma^{(1)}$ of $F_\Pi(\Gamma^{(1)},\ldots,\Gamma^{(t)})$. Let $u,z\in V(\Gamma^{(1)})$ be vertices at distance 2 in $\Gamma^{(1)}$ (which exists as $\Gamma^{(1)}$ is not complete). Note that this means $u$ and $z$ belong to different perfect 1-codes that partition $\Gamma^{(1)}$. Let $m,n$ be the distinct integers such that $u \in H^{(1)}_{m},z \in H^{(1)}_{n}$. Since $u,z$ are at distance 2 in $\Gamma^{(1)}$, they have at least one common neighbour in $\Gamma^{(1)} \setminus (H^{(1)}_{m} \cup H^{(1)}_{n})$. Moreover, in $F_\Pi(\Gamma^{(1)},\ldots,\Gamma^{(t)})$, the vertices $u$ and $z$ have one common neighbour in $H^{(1)}_{m}$ and one common neighbour in $H^{(1)}_{n}$. Thus, $u$ and $z$ have at least three common neighbours in $F_\Pi(\Gamma^{(1)},\ldots,\Gamma^{(t)})$.

We have shown that in $F_\Pi(\Gamma^{(1)},\ldots,\Gamma^{(t)})$,  there is a pair of non-adjacent vertices with exactly $2$ common neighbours and a pair of non-adjacent vertices with at least $3$ common neighbours. Therefore, $F_\Pi(\Gamma^{(1)},\ldots,\Gamma^{(t)})$ is not a co-edge-regular graph.
\end{proof}

In Theorem \ref{GeneralConstruction}, we can use an arbitrary $(t-1)$-tuple of permutations $\Pi$. In the following result, we show that a certain WQH-switching in our graphs is equivalent to a certain change of tuple $\Pi$. We note that for $t=1$, our choice of partition also satisfies the criteria for a switching in Lemma \ref{WQHswitching}, but the switching operation does not change the graph.

\begin{proposition}\label{Switching}
Let $\Gamma^{(1)},\dots,\Gamma^{(t)}$ be edge-regular graphs with parameters $(v,k,\lambda)$ and let $\Pi =(\pi_{2},\dots,\pi_{t})$ be a $(t-1)$-tuple of permutations from
$\text{Sym}(\{1, \ldots, v/a \})$ when $t\geq 2$.

Further, suppose that each graph $\Gamma^{(\ell)}$ has a partition of its vertices into perfect $1$-codes $H_{1}^{(\ell)},\dots,H_{v/a}^{(\ell)}$, each of size $a$, where $a=(\lambda+2)/t.$

Then for any non-empty subset $I\subseteq \{2,\dots,t\}$ and distinct $i,j\in \{1,\dots,v/a\}$, the partition $$C_1:= H^{(1)}_{i}\cup \left(\bigcup_{\ell\in I} H^{(\ell)}_{\pi_{\ell}(i)}\right),$$
$$C_2:=H^{(1)}_{j}\cup \left(\bigcup_{\ell\in I} H^{(\ell)}_{\pi_{\ell}(j)}\right),$$
and
$$D:=V(F_\Pi(\Gamma^{(1)},\ldots,\Gamma^{(t)})) \setminus (C_1\cup C_2 )$$
satisfies the conditions of Lemma \ref{WQHswitching}. 

%p. 6, l. 31–35. Should there be no πk’s be in this definition? If I understand the arguments in the proof correctly, you want C1 and C2 each to be a subset of one of the 1-regular cliques. 

%SGComment 2. As you have modified it, the words "containing 1" should be removed. 

%CHANGE - Proposition 1, comment meant to say we were missing pi_k from definition of C_1 C_2. - DONE

In fact, we have the equality $(F_\Pi(\Gamma^{(1)},\ldots,\Gamma^{(t)}))' = F_{\Pi'}(\Gamma^{(1)},\ldots,\Gamma^{(t)})$,
where 
\begin{equation*}
(\Pi')_r =\begin{cases}
\pi_r & \text{if } r\in I \\
(i~j)\circ \pi_r & \text{if } r\not\in I
\end{cases}
\end{equation*}
\end{proposition}
\begin{proof}
	Let $\Gamma=F_\Pi(\Gamma^{(1)},\ldots,\Gamma^{(t)})$. We have the following cases for vertices $u\in D$:
	\begin{enumerate}
		\item $u\in H_{\pi_m(i)}^{(m)}$ for some $m\not\in I$. By the definition of our construction,\linebreak $\Gamma(u)\cap (C_1 \cup C_2) = C_1$.
		\item $u\in H_{\pi_m(j)}^{(m)}$ for some $m\not\in I$. By the definition of our construction,\linebreak $\Gamma(u)\cap (C_1 \cup C_2) = C_2$.
		\item $u\in H_{\pi_m(h)}^{(m)}$ for some $m,h$ where $m\not\in I,h\not\in \{i,j\}$. This means that $u$ is not from any $\Gamma^{(\ell)}$ for $\ell\in I$, or the cliques containing $C_1$ or $C_2$ that are created by our construction. Therefore, $|\Gamma(u)\cap C_1|=|\Gamma(u)\cap C_2|=0$.
		\item $u\in H_h^{(m)}$ for some $m\in I,h\not\in \{i,j\}$. As the set $H_h^{(m)}$ is a perfect 1-code in $\Gamma^{(m)}$, we have $|\Gamma(u)\cap C_1|=|\Gamma(u)\cap C_2|=1$.
	\end{enumerate}
	We have shown that the partition satisfies the conditions of Lemma \ref{WQHswitching}. From the cases for $u\in D$ above and the definition of WQH-switching, it also follows that $(F_\Pi(\Gamma^{(1)},\ldots,\Gamma^{(t)}))' = F_{\Pi'}(\Gamma^{(1)},\ldots,\Gamma^{(t)})$
	with the stated $\Pi'$.
\end{proof}

Given a certain graph property, it is common in spectral graph theory to ask if this property is determined by the spectrum of a graph with this property. In Proposition \ref{Switching}, we have seen that we have given a method for constructing cospectral Neumaier graphs. Therefore, we can investigate if (strictly) Neumaier graphs with a 1-regular clique are determined by their spectra. The following shows that this is not true for strictly Neumaier graphs with 1-regular cliques.

From Proposition \ref{Switching}, we can see that for any two $(t-1)$-tuples of elements of $\text{Sym}(\{1,\dots,v/a\})$, $\Pi,\Pi'$, we can obtain $F_{\Pi'}(\Gamma^{(1)},\ldots,\Gamma^{(t)})$ from  $F_\Pi(\Gamma^{(1)},\ldots,\Gamma^{(t)})$ by applying a series of appropriate WQH-switchings. As the switching operation does not change the spectrum of the resulting graph, we get the following result.

\begin{corollary}
	For any $\Pi,\Pi'$, $(t-1)$-tuples of elements of $\text{Sym}(\{1,\dots,v/a\})$, the graphs
    $F_\Pi(\Gamma^{(1)},\ldots,\Gamma^{(t)})$ and 	$F_{\Pi'}(\Gamma^{(1)},\ldots,\Gamma^{(t)})$ are cospectral.  
\end{corollary} 

It would be interesting to investigate how many pairwise non-isomorphic graphs can be constructed using our construction. In doing so, we may find a prolific construction of cospectral strictly Neumaier graphs. Although this has not been investigated in detail, we have already observed several pairwise non-isomorphic graphs with relatively small order.

\begin{corollary}\label{Neumaier24AreCospectral}
	There are four strictly Neumaier graphs with parameters $(24,8,2;1,4)$ which are cospectral. 
\end{corollary}
\begin{proof}
	This can be shown by applying iterations of WQH-switchings to a given graph from Example \ref{ex:icos}. This is also easily verified by direct calculation of the spectra of the four graph from Example \ref{ex:icos}. 
	
	The adjacency matrices of these graphs can be found in \cite[Section 4.4.1]{E20}. The spectra of the four strictly Neumaier graphs are all equal to
	$\left[(-1-\sqrt{5})^6, (-2)^5, (-1+\sqrt{5})^6, 2^5, 4^1, 8^1 \right]$.	
\end{proof}

\section{Some examples of strictly Neumaier graphs given by the general construction}\label{sec:Examples}

In this section, we give several small strictly Neumaier graphs which can be constructed using Theorem \ref{GeneralConstruction}. Some of these graphs are new, and have not been found as part of an infinite family of strictly Neumaier graphs at the time of writing this paper. In these graphs, we find our underlying edge-regular graphs from various geometric and algebraic constructions.

%p. 8, l. 17. The webpage address could be moved to the bibliography, and cited here. 

%CHANGE - References D. Panasenko - DONE

\subsection{Construction given by a pair of icosahedrons}\label{ex:icos}
The icosahedral graph is an edge-regular graph with parameters $(12,5,2)$ that admits a partition into six perfect 1-codes of size $a = 2$. Thus, we can use $t = (\lambda+2)/a= 2$ copies of the icosahedral graph in the general construction to produce four pairwise non-isomorphic strictly Neumaier graphs (depending on the choice of the permutation $\pi_2$) with parameters $(24,8,2;1,4)$. Note that these four graphs were found as Cayley-Deza graphs in \cite{GS14}, and listed online at \cite{P22}.

One of these graphs is constructed in \cite{GK19}, and no other strictly Neumaier graphs with less than $28$ vertices are found in \cite{GK18} or \cite{GK19}. 

%p. 8, l. 24. I think it must be x2 here.

%CHANGE - Section 4.4 Paragraph 1 - DONE
\subsection{Construction given by a pair of dodecahedrons}
Take a pair of dodecahedral graphs $\Gamma_1$ and $\Gamma_2$ and consider the natural matching between their vertices. Note that the dodecahedral graph has diameter 5 and admits a partition into ten perfect 2-codes of size 2. For every vertex $x_1 \in \Gamma_1$ and its matched vertex $x_2 \in \Gamma_2$, connect $x_1$ with the six vertices of $\Gamma_2$ that are at distance 2 from $x_2$ and connect $x_2$ with the six vertices of $\Gamma_1$ that are at distance 2 from $x_1$. We obtain an edge-regular graph with parameters $(40,9,2)$ that admits a partition into ten perfect 1-codes of size $a = 4$, where every perfect 1-code is a union of matched perfect 2-codes in the original dodecahedral graphs. We then apply the general construction to the edge-regular graph with parameters $(40,9,2)$, which gives a strictly Neumaier graph with parameters $(40,12,2;1,4)$.

As $k-\lambda=12-2=10$ is not a prime power, the parameters do not meet the necessary conditions of the construction in \cite{GK18}. The resulting graph has a unique spread of 1-regular cliques. Removing the edges of the cliques of the spread does not result in a distance-regular graph, so this graph does not come from the construction in \cite{GK19}.

\subsection{A graph on 65 vertices}\label{graph65}
Consider the Cayley graph $Cay(\mathbb{Z}^+_{65},\{2^i~mod ~65 : i \in \{0,\ldots, 11\}\}$, which is edge-regular with parameters $(65,12,3)$ and admits a partition into perfect 1-codes of size $a = 5$ given by the cosets of the subgroup $\{0,13,26,39,52\}$. We then apply the general construction, which gives a strictly Neumaier graph with parameters $(65,16,3;1,5)$. 

As the resulting graph has an odd number of vertices, the parameters do not meet the necessary conditions of the construction in \cite{GK18}. The resulting graph also has a unique spread of 1-regular cliques. Removing the edges of the cliques of the spread does not result in a distance-regular graph, so this graph does not come from the construction in \cite{GK19}.   

\subsection{Construction given by the 6-regular triangular grid}
Eisenstein integers are the complex numbers of the form $\mathbb{Z}[\omega] = \{b+c\omega : b,c \in \mathbb{Z}\}$, where $\omega = \frac{-1+i\sqrt{3}}{2}$. They form a ring with respect to usual addition and multiplication.
The norm mapping $N:\mathbb{Z}[\omega]\mapsto \mathbb{N} \cup \{0\}$ is defined as follows. For an Eisenstein integer $b+c\omega$, $N(b+c\omega) = b^2+c^2-bc$ holds. The norm mapping $N$ is known to be multiplicative. It is well-known that $\mathbb{Z}[\omega]$ forms an Euclidean domain (in particular, a principal ideal domain). 
The units of $\mathbb{Z}[\omega]$ are $\{\pm 1, \pm\omega, \pm\omega^2\}$. 
The natural geometrical interpretation of Eisenstein integers is the 6-regular triangular grid in the complex plane. If it does not lead to a contradiction, we use the same notation $\mathbb{Z}[\omega]$ for the triangular grid. 

%p. 8, l. 54 and 55. The statements that I is a perfect 1-code in the grid and a subgroup of index 7 could use some more explanation. One does not see these things immediately.

The grid $\mathbb{Z}[\omega]$ has exactly six elements of norm 7; these are $\{\pm(1+3\omega),\pm(3+2\omega),\pm(2-\omega)\}$. Consider the ideal $I$ generated by an element of norm 7 (say, by the element $2-\omega$). Note that as $I$ is generated by an element of norm 7, $I$ is an additive subgroup of index 7 in $\mathbb{Z}[\omega]$; we denote it by $I^+$.

In Figure \ref{Ideal}, we represent a subset of the grid $\mathbb{Z}[\omega]$ by points, and distinct points are connected by an edge when they differ by an element of $\{\pm 1, \pm\omega\, \pm(1+\omega)\}$. Elements of $I$ are in blue, and the balls of radius 1 centred at elements of $I$ are highlighted by bold edges between two members of the same ball. Figure \ref{Ideal} illustrates that elements of $I$ form a perfect 1-code in the triangular grid.

The seven cosets $\mathbb{Z}[\omega]/I^+$ give a partition of the triangular grid into seven perfect 1-codes formed by vertices of the same coset, represented by vertices with the same colour in Figure \ref{Partition}. For example, the dark blue vertices are formed by the elements of the ideal $I$. 
Consider a quadruple (block) of the balls of radius 1 centred at dark blue vertices in Figure \ref{Block}; these balls have 28 vertices. Take the following two additive subgroups of $\mathbb{Z}[\omega]$:
$$
T_1:=\{2(-2+\omega)x+14y~|~x,y\in \mathbb{Z}\},
$$
$$
T_2:=\{(5+\omega)x+28y~|~x,y\in \mathbb{Z}\}.
$$

%p. 9, l. 17. This block of four balls that gives two tessellations should be described. Or present a picture. I know detailed pictures are available in the arxiv paper, but including one here, is acceptable.

Since $-2+\omega$, $7$ and $5+\omega$ are divisible by $2-\omega$, a generator of $I$, we have that $T_1$ and $T_2$ are subgroups in $I^+$ (see Figures \ref{T1} and \ref{T2}).Note that there exists a block of four balls of radius 1 such that the additive shifts of the block by the elements of $T_1$ and $T_2$ give two tessellations of $\mathbb{Z}[\omega]$ (see Figures \ref{T1} and \ref{T2}, where tiles consist of all points lying on and within an area encompassed by bold edges). 
Consider the quotient groups $$G_1 := \mathbb{Z}[\omega]/T_1$$ and $$G_2:=\mathbb{Z}[\omega]/T_2,$$
where $G_1 \cong \mathbb{Z}_2 \oplus \mathbb{Z}_{14}$ and $G_2 \cong \mathbb{Z}_{28}$.
Define two Cayley graphs 
$$\Delta_1:=Cay(G_1,\{\pm(1+T_1), \pm(\omega+T_1), \pm(\omega^2+T_1)\}),$$
$$\Delta_2:=Cay(G_2,\{\pm(1+T_2), \pm(\omega+T_2), \pm(\omega^2+T_2)\}).$$
Note that $\Delta_1$ and $\Delta_2$ can be interpreted as quotient graphs of the triangular grid by $T_1$ and $T_2$, respectively. 
Each of the graphs $\Delta_1$ and $\Delta_2$ is edge-regular with parameters $(28,6,2)$ and admits a partition into perfect 1-codes of size $a = 4$; these partitions are given by the original partition of the triangular grid into perfect 1-codes (see Figures \ref{Delta1} and \ref{Delta2}). We then apply the general construction, which gives two strictly Neumaier graphs with parameters $(28,9,2;1,4)$. 

In \cite[Example 1]{GK18} exactly one strictly Neumaier is explicitly constructed with generating group $G_1$, which is exactly the graph $\Delta_1$. Each resulting graph also has a unique spread of 1-regular cliques. Removing the edges of a spread does not result in a distance-regular graph, so the graph does not come from the construction in \cite{GK19}.

%P9 L46-48 "We then apply the general construction to these graphs, which gives at least eight strictly Neumaier graphs with parameters (78, 17, 4; 1, 6)."
%Comment: The authors should describe why the graphs with parameters (78, 17, 4; 1, 6) are present both in Table 1 and in Table 2. 

% SGcomment: The answer is already in the paper and just should be emphasized: "The root lattices A3 and D3 are both isomorphic to the tetrahedral-octahedral honeycomb."

%CHANGE- Paragraph before Tables, describing tables - DONE

%p. 9, Subsection 4.5. I do believe the discussion is similar, but I think you could include at least the description of the tetrahedral-octahedral honeycomb, and maybe the partition. This would be reader-friendly. 

% SGcomment: We started to deal with the tetrahedral-octahedral honeycomb as with a geometric object. Then we realised that it can be described in terms of root lattices (an algebraic description). I have no idea how to give a reader-friendly description of the tetrahedral-octahedral honeycomb. Maybe, just "look at the picture". Moreover, I believe the easiest way to give a description of the subgroup perfect code in this lattice is to give generators of the subgroup and show that the subgroup gives a prefect code. I hoped we can do this in a next paper, which is supposed to pay more attention to the grids.

%CHANGE - Section 4.5 - DONE

\subsection{Graphs on 78 vertices from the tetrahedral-octahedral honeycomb}\label{ss:honeycomb}

The tetrahedral-octahedral honeycomb is a tiling of the 3-dimensional Euclidean space with an alternating pattern of tetrahedra and octahedra. The skeleton of this tiling defines an infinite graph which can be described as a 3-dimensional sublattice $L$ of $\mathbb{Z}^3$. Similar to the case of the triangular grid, we find a subgroup $J$ of index 13 in $L$, for which the balls of radius 1 around the elements of $J$ give a tiling of $L$. This gives a partition of the skeleton into thirteen perfect 1-codes. Using computational tools, we then take different quotients which preserve the partition, and find at least eight edge-regular graphs with parameters $(78,12,4)$. Applying the general construction to these graphs yields at least eight strictly Neumaier graphs with parameters $(78,17,4;1,6)$. 

Alternatively, these graphs can be constructed as circulant graphs. Consider the subsets $S_i$ of the group $\mathbb{Z}_{78}$, for $i\in\{1,2,\dots,8\}$, where
\begin{align*}
S_1 &=\{ 1, 3, 4, 15, 18, 19, 59, 60, 63, 74, 75, 77 \},\\
S_2 &=\{ 1, 4, 5, 15, 16, 20, 58, 62, 63, 73, 74, 77 \},\\
S_3 &=\{ 1, 15, 16, 19, 34, 35, 43, 44, 59, 62, 63, 77 \},\\
S_4 &=\{ 1, 6, 15, 16, 21, 22, 56, 57, 62, 63, 72, 77 \},\\
S_5 &=\{ 1, 17, 18, 28, 29, 32, 46, 49, 50, 60, 61, 77 \},\\
S_6 &=\{ 1, 7, 8, 29, 30, 37, 41, 48, 49, 70, 71, 77 \},\\
S_7 &=\{ 1, 3, 4, 20, 21, 24, 54, 57, 58, 74, 75, 77 \},\\
S_8 &=\{ 1, 6, 17, 18, 23, 24, 54, 55, 60, 61, 72, 77 \}.
\end{align*}
Each of the edge-regular graphs with parameters $(78,12,4)$ we construct from the honeycomb is isomorphic to $Cay(\mathbb{Z}^+_{78},S_i)$ for some $i$. Furthermore, the strictly Neumaier graphs found by applying our construction to these graphs are exactly the graphs $Cay(\mathbb{Z}^+_{78},S_i\cup H^*)$, where $H^*$ is the set of nonidentity elements of the subgroup of $\mathbb{Z}^+_{78}$ of index $13$.

As each resulting graph has regular cliques with sizes not divisible by 4, the parameters do not meet the necessary conditions of the explicit constructions in \cite{GK18}. Each resulting graph also has a unique spread of 1-regular cliques. Removing the edges of the cliques of the spread does not result in a distance-regular graph, so the graph does not come from the construction in \cite{GK19}.

\section{Further problems: Infinite edge-regular lattices}\label{sec:FurtherProblems}

In Section \ref{sec:Examples} we have seen several strictly Neumaier graphs constructed using infinite lattices which are edge-regular. In this section, we give certain families of infinite edge-regular lattices, and find new strictly Neumaier graphs in certain cases. We also present open problems relevant to the construction of new infinite families of strictly Neumaier graphs from these lattices.

In the following, we consider countably infinite graphs with vertices consisting of elements of the vector space $\mathbb{R}^n$, for some integer $n\geq 3$. The elements of $\mathbb{R}^n$ are called \emph{($n$-dimensional) vectors}, and we identify the elements with their coordinates with respect to the standard basis of $\mathbb{R}^n$. 

Let $x\in \mathbb{R}^n$. For a set $A\subseteq \mathbb{R}$, the vector $x$ is an \emph{$A$-vector} if the value of all of its entries lie in $A$. The \emph{weight} of $x$ is the number of its non-zero entries.

Let $n \ge 3$ be a positive integer and let $m$ be an even positive integer.
Let $S_{n,m}^{(1)}$ denote the set of all $n$-dimensional $\{1,-1,0\}$-vectors of weight $m$ whose sum of coordinates is zero.
Let $S_{n,m}^{(2)}$ denote the set of all $n$-dimensional $\{1,-1,0\}$-vectors of weight $m$.
Let $G_{n,m}^{(1)}$ and $G_{n,m}^{(2)}$ be the groups generated by $S_{n,m}^{(1)}$ and $S_{n,m}^{(2)}$ respectively.

In fact, for any positive even integer $m$ and any integer $n$ such that $n \ge m+1$, we have $G_{n,m}^{(1)}=G_{n,2}^{(1)}$ and $G_{n,m}^{(2)}=G_{n,2}^{(2)}.$ From now on, we let $$G_{n}^{(1)}:=G_{n,2}^{(1)},$$
$$G_{n}^{(2)}:=G_{n,2}^{(2)},$$
and define graphs $$\Gamma_{n,m}^{(1)}:=Cay(G_{n}^{(1)},S_{n,m}^{(1)})$$ and
$$\Gamma_{n,m}^{(2)}:=Cay(G_{n}^{(2)},S_{n,m}^{(2)}).$$

A graph $\Gamma$ 
with infinitely many vertices is \emph{edge-regular} with \emph{parameters} $(k,\lambda)$ if it is $k$-regular and each pair of adjacent vertices have exactly $\lambda$ common neighbours.

%P10 L56 "In our notation, we have $C^b_a = \binom{a}{b}$."
%Comment: $\binom{a}{b}$ is a conventional notation for the binomial coefficient. I recommend the authors to use $\binom{a}{b}$. DONE

%p. 10, l. 56. I would advise to use the ab-notation, which is standard, unlike the Cba-notation. 

%SGComment: I agree with the referee.

%CHANGE - Proposition 2 - DONE
 
In the following, we show that  $\Gamma_{n,m}^{(1)}$ and $\Gamma_{n,m}^{(2)}$ are infinite edge-regular graphs, and give the parameters of these graphs in terms of binomial coefficients $\binom{a}{b}$. 

\begin{proposition} For any positive even integer $m$ and any integer $n$ such that $n \ge m+1$, the following statements hold.
\begin{enumerate}
    \item The graph $\Gamma_{n,m}^{(1)}$ is an induced subgraph in $\Gamma_{n,m}^{(2)}$.

%p. 11, l. 28. I think it is ‘cleaner’ to write the two final factors in the product as n−m i n−m−i m/2−i. Actually, λ1 = n−m m/2 Pm/2 i=0 m/2 i 3, which helps to understand Remark 2.  

% SGComment 2. I agree that the suggested way to write the two final factors is cleaner. Also, the equality you added to the statement is not what the referee suggested (it seems you missed a factor). I suggest to add a proof for this equality (I spent 30 minutes to find a "standard binomial identity" to show the equality; I even tried the first Strehl identity https://mathworld.wolfram.com/StrehlIdentities.html but it was not necessary).

%CHANGE - Proposition 2 statement, significant change from reviewer - DONE

    \item $\Gamma_{n,m}^{(1)}$ is an infinite edge-regular graph with parameters $(k_1,\lambda_1)$, such that $$k_1 = \binom{n}{m}\binom{m}{m/2},$$  
    \begin{eqnarray*}
        \lambda_1 &=& \sum\limits_{i=0}^{m/2}\binom{m/2}{i}\binom{m/2}{m/2-i}\binom{n-m}{i}\binom{n-m-i}{m/2-i}\\
                  &=& \binom{n-m}{m/2}\sum\limits_{i=0}^{m/2}\binom{m/2}{i}^3.
    \end{eqnarray*}
%p. 11, l. 30. (k2, λ2). 

%CHANGE - Proposition 2 statement - DONE
    \item $\Gamma_{n,m}^{(2)}$ is an infinite edge-regular graph with parameters $(k_2,\lambda_2)$, such that $$k_2 = 2^m\binom{n}{m},$$  $$\lambda_2 = \binom{m}{m/2}\binom{n-m}{m/2}2^{m/2}.$$
\end{enumerate}
\end{proposition}

%p. 11, l. 36. I think it is important to point out that G(1)n ∩S(2)n,m = S(1)n,m, to conclude that it is an induced subgraph. 

%CHANGE - Proposition 2 proof - DONE

\begin{proof}
The elements of $G_{n}^{(1)}$ form a subset in the set of elements of $G_{n}^{(2)}$ and $G_{n}^{(1)}\cap S_{n,m}^{(2)} = S_{n,m}^{(1)}$, so $\Gamma_{n,m}^{(1)}$ is an induced subgraph of $\Gamma_{n,m}^{(2)}$. The formulas for parameters $k_1, k_2,$ and $\lambda_2$ can be obtained by applying standard counting arguments. The first formula for $\lambda_1$ is also found using standard counting arguments, and equality with the second formula is proven by using the definition of the binomial coefficients to manipulate the expressions in the first formula. 
\end{proof}
\begin{remark}
Note that if $n < 3m/2$, then $\lambda_1 = 0$ and $\lambda_2 = 0$. Otherwise, $\lambda_1 > 0$ and $\lambda_2 > 0$. 
\end{remark}

%p. 11, l. 46–49. Here, you should elaborate a little bit. Up to now, you were dealing with two families of infinite edge-regular graphs. Now, you mention root lattices (without defining them), say there are isomorphic to grids (how), and then all of sudden you find strictly Neumaier graphs. In the caption of the table you mention quotients (which you do not mention here), but I guess these quotients are non-trivial. 

% SGcomment: The same as in the previous comment. We can give all the necessary definitions, but further details are to be realised. In this paper, it is better to mention that we were able to find such examples (using a computer), that we think that this topic is wider than the topic of this paper and that we plan to investigate the grids in all details in next papers.

%CHANGE - between remark 3 and problem 1  - DONE

%p. 11, l. 46. “The A2 ...” 

%CHANGE - Remark 3 - DONE

\begin{remark}\label{rootSystems}
The generating sets $S_{n,2}^{(1)}$ and $S_{n,2}^{(2)}$ are known as root systems $A_{n-1}$ and $D_n$, respectively (see \cite[Chapter 8]{BH12}). In particular, the root lattice $A_2$ is isomorphic to the 6-regular triangular grid, and the root lattices $A_3$ and $D_3$ are both isomorphic to the tetrahedral-octahedral honeycomb.
\end{remark}

Now we have an approach to construct Neumaier graphs using a computer, which is similar the approaches in Sections \ref{ss:6reggrid} and \ref{ss:honeycomb}. In this process we start with $G=G_{n}^{(i)}$ and $\Gamma=\Gamma_{n,m}^{(i)}$, for some integer $n>2$, positive and even integer $m$, and $i\in \{1,2\}$. Then we find a subgroup $I$ of $G$ which is a perfect 1-code in $\Gamma$, and has finite index determined by what parameter set we want our resulting Neumaier graph to have. A subgroup $T$ of $I$ is then chosen which has the property that $\Gamma/T$ is still edge-regular, and that the cosets of $I/T$ partition $\Gamma/T$ into perfect 1-codes. We can then apply our general construction to find a Neumaier graph, with  isomorphism class that may depend on our choice parameters $n,m$ and $i$, and the groups $I$ and $T$.

%P11 L49 Should be "$\Gamma^{(1)}_{n,2}$ and $\Gamma^{(2)}_{n,2}$" instead of "$\Gamma^{(1)}_{n,m}$ and $\Gamma^{(2)}_{n,m}$". 

%CHANGE - p12 Paragraph describing tables - DONE

In Tables \ref{tab:gamma1} and \ref{tab:gamma2}, we present the number of cases for which we find strictly Neumaier graphs by taking quotients of the graphs $\Gamma_{n,2}^{(1)}$ and $\Gamma_{n,2}^{(2)}$, respectively. The first column of the tables give the corresponding value of $n$. The second column gives the Neumaier graph parameters of the graphs we find through the construction. The last column gives the number of pairwise non-isomorphic strictly Neumaier graphs we find from the construction. Note that by Remark \ref{rootSystems}, the entry for $n=4$ in Table \ref{tab:gamma1} will coincide with the entry for $n=3$ in Table \ref{tab:gamma2}.

\begin{table}
\centering
\begin{tabular}{|c|c|c|}
  \hline
  $n$ & parameters of SNG & $\#$ \\
  \hline
  3 & $(28,9,2;1,4)$ & 2 \\
  \hline
  4 & $(78,17,4;1,6)$ & $\ge 8$ \\
  \hline
  5 & $(168,27,6;1,8)$ & $\ge 12$ \\
  \hline
  6 & $(310,39,8;1,10)$ & $\ge 1$ \\
  \hline
\end{tabular}
\caption{Number of strictly Neumaier graphs from quotients of $\Gamma_{n,2}^{(1)}$.}\label{tab:gamma1}
\end{table}

\begin{table}
\centering
\begin{tabular}{|c|c|c|}
  \hline
  $n$ & parameters of SNG & $\#$ \\
  \hline
  3 & $(78,17,4;1,6)$ & $\ge 8$ \\
  \hline
  4 & $(250,33,8;1,10)$ & $\ge 16$ \\
  \hline
\end{tabular}
\caption{Number of strictly Neumaier graphs from quotients of $\Gamma_{n,2}^{(2)}$.}\label{tab:gamma2}
\end{table}

We have not been able to find more examples of perfect codes and quotients of $\Gamma_{n,2}^{(1)}$ and $\Gamma_{n,2}^{(2)}$ that lead to strictly Neumaier graphs. We plan to determine all conditions necessary for our approach to construct Neumaier graphs and apply our ideas to infinite families of graphs. However, we do not believe this is within the scope of the current paper. Therefore, we ask the following.

\begin{problem}
What strictly Neumaier graphs can be obtained from quotients of infinite edge-regular graphs $\Gamma_{n,m}^{(1)}$ and $\Gamma_{n,m}^{(2)}$?
\end{problem}

We can also use two infinite edge-regular graphs to get a new infinite edge-regular graph by taking the Cartesian product of the graphs (for the definition of Cartesian product of graphs, see \cite{BH12}, for example).

\begin{proposition}\label{CartProd}
Let $\Gamma_1$ and $\Gamma_2$ be two infinite edge-regular graphs with parameters $(k_1,\lambda)$ and $(k_2,\lambda)$, respectively. Then the Cartesian product of $\Gamma_1$ and $\Gamma_2$ is an edge-regular graph with parameters $(k_1+k_2, \lambda)$. 
\end{proposition}
%p. 12, l. 26. There is a double square. 

%CHANGE - Proposition 3 - DONE
\begin{proof}
It follows from the definition of Cartesian product.
\end{proof}
Consider the Cartesian product of two 6-regular triangular grids; the resulting infinite graph is edge-regular with parameters (12,2). This graph has a perfect 1-code, and there exists an edge-regular quotient graph with parameters $(52,12,2)$. We then apply the general construction to this graph, which gives a strictly Neumaier graph having parameters $(52,15,2;1,4)$ and isomorphic to the graph constructed in \cite[Theorem 3.6]{GK18}. As we have seen this example using Cartesian products, we ask the following.

\begin{problem}
What strictly Neumaier graphs can be obtained from quotients of Cartesian products of infinite edge-regular graphs?
\end{problem}

\section*{Acknowledgements} \label{Ack}
The work of R.~J.~Evans, E.~V.~Konstantinova, and A.~D.~Mednykh was supported by the Mathematical Center in Akademgorodok, under agreement No. 075-15-2022-281 
with the Ministry of Science and High Education of the Russian Federation. S.~Goryainov also thanks Mathematical Center in Akademgorodok for organising his visit in Akademgorodok in July 2021. Finally, S.~Goryainov and E.~V.~Konstantinova thank D.~Krotov for useful discussions.

% Needed in overleaf?
%\section*{References}

\begin{landscape}
	\begin{figure}[htbp]
		\includegraphics[scale=0.15]{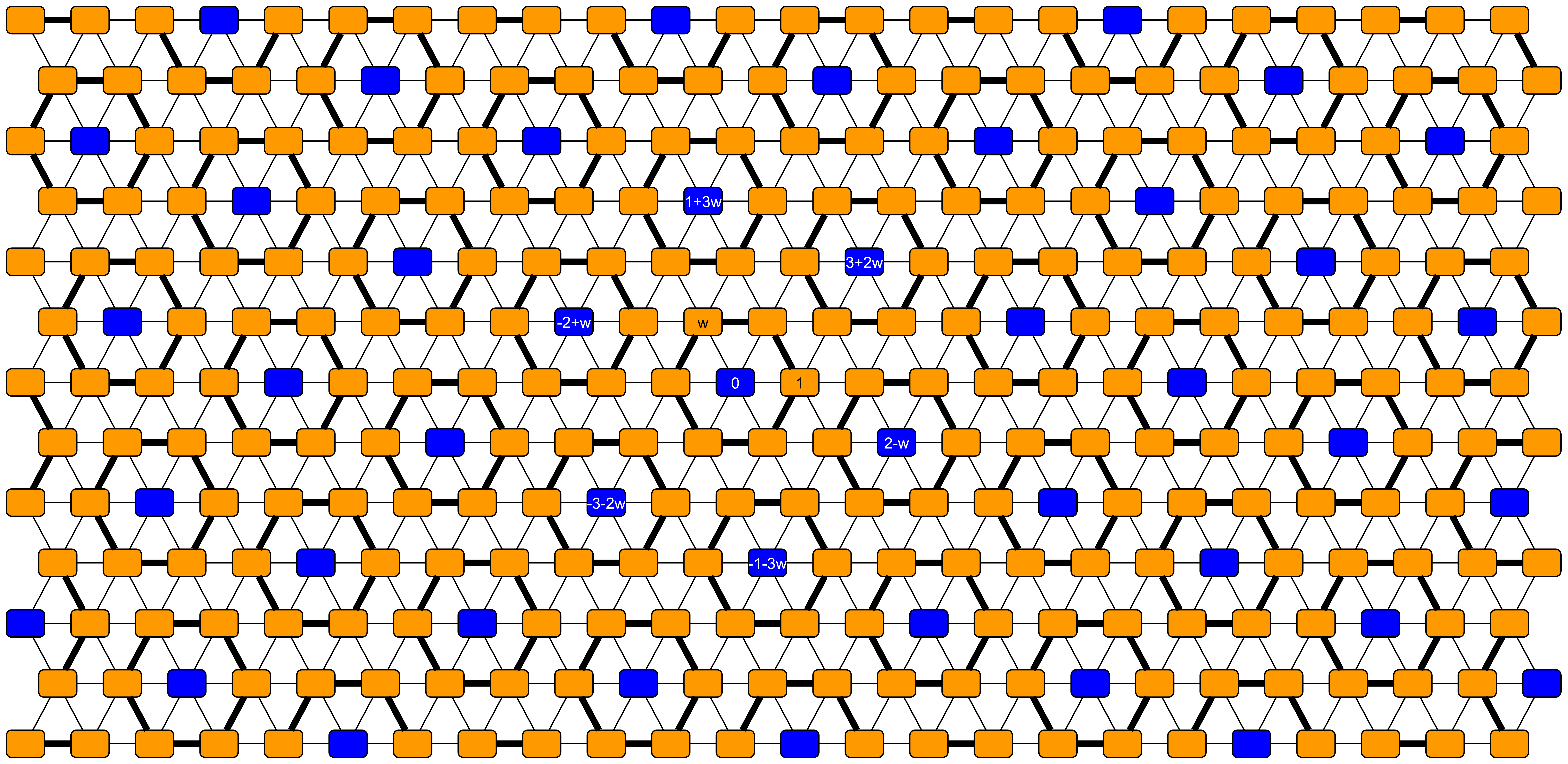}
		\caption{The ideal generated by an element of norm 7}
		\label{Ideal}
	\end{figure}
\end{landscape}

\begin{landscape}
	\begin{figure}[htbp]
		\includegraphics[scale=0.15]{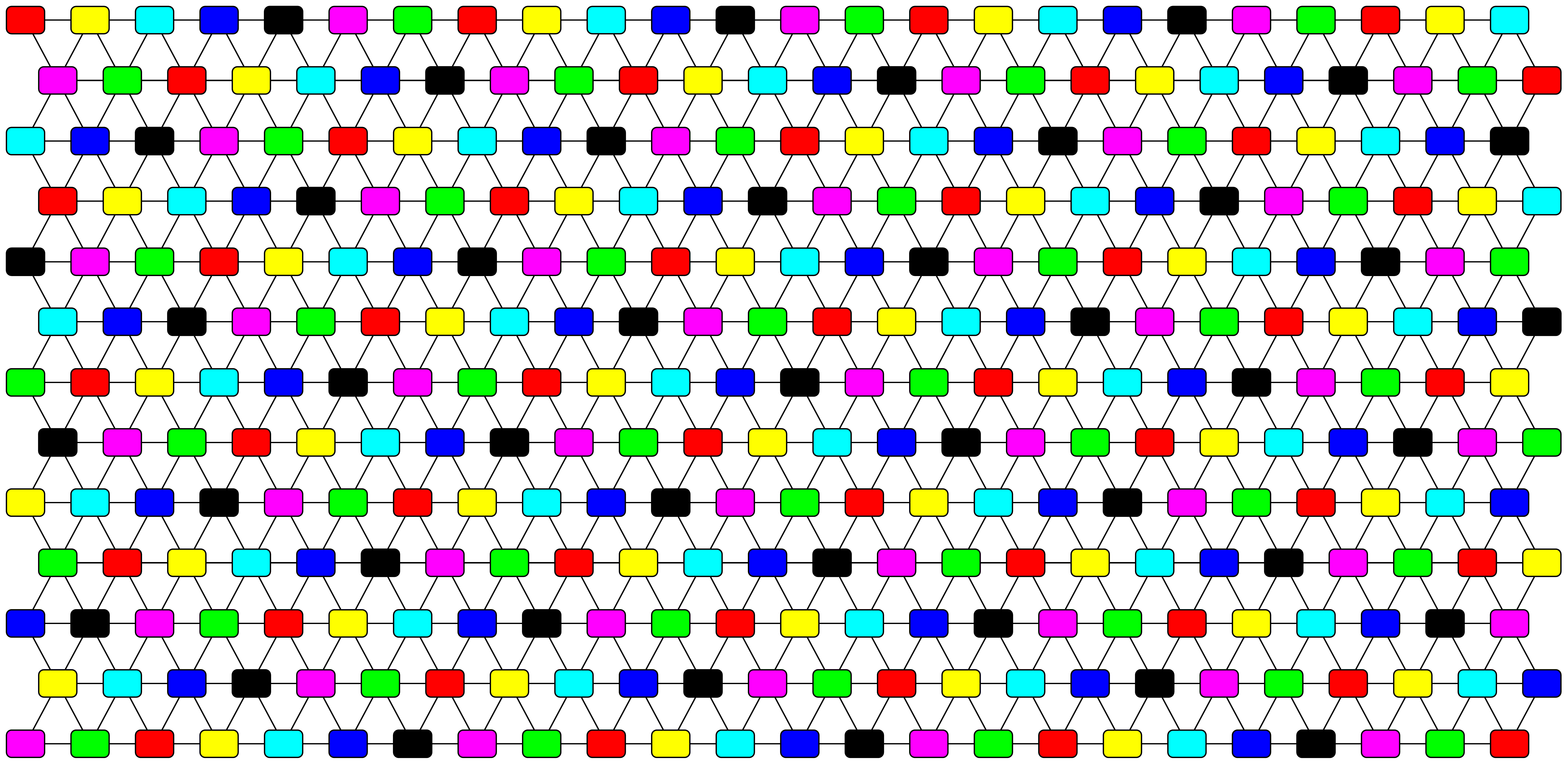}
		\caption{Partition into perfect 1-codes}
		\label{Partition}
	\end{figure}
\end{landscape}

\begin{landscape}
	\begin{figure}[htbp]
		\includegraphics[scale=0.15]{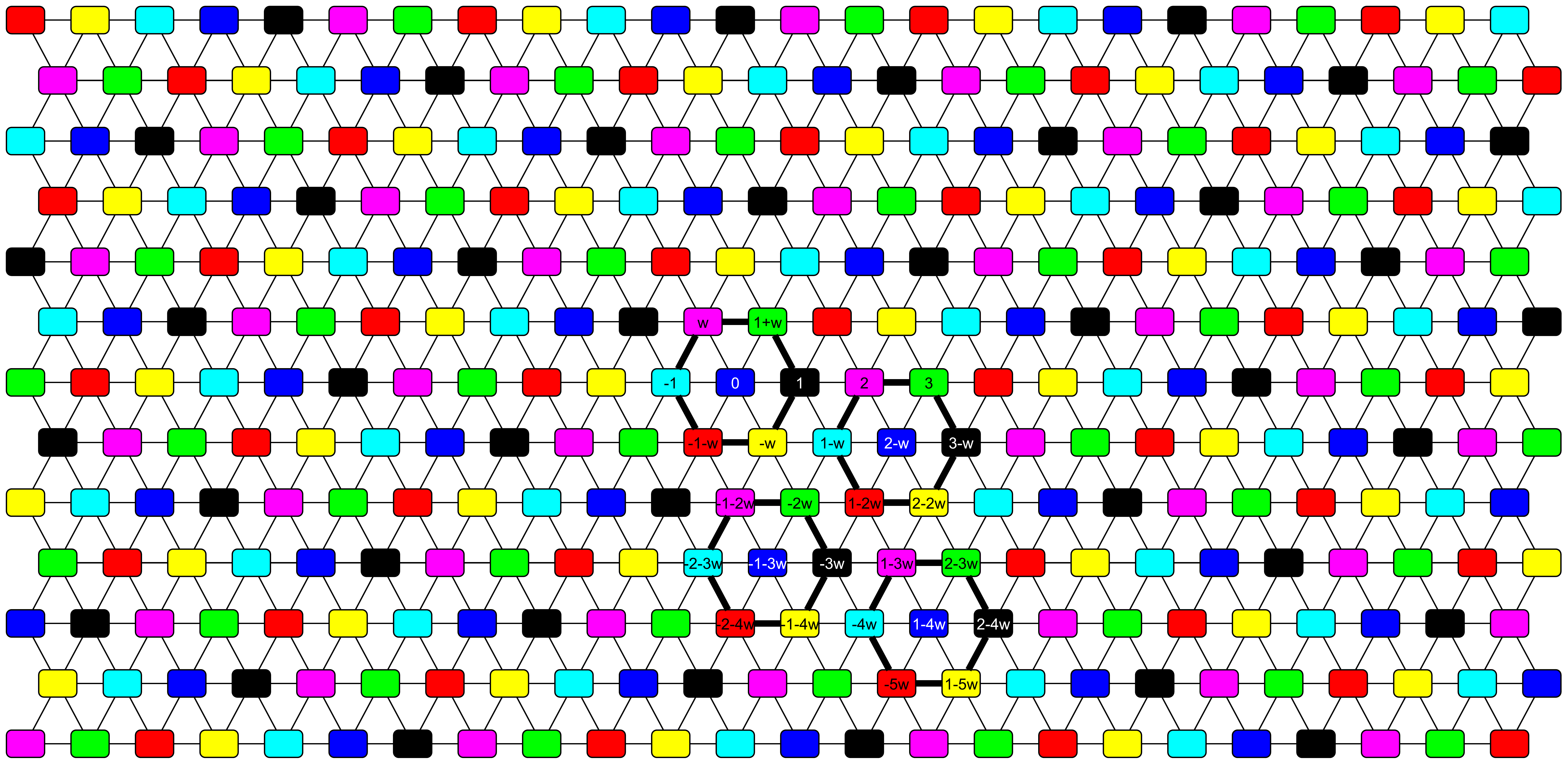}
		\caption{The block of balls of radius 1 centred at dark blue vertices}
		\label{Block}
	\end{figure}
\end{landscape}

\begin{landscape}
	\begin{figure}[htbp]
		\includegraphics[scale=0.15]{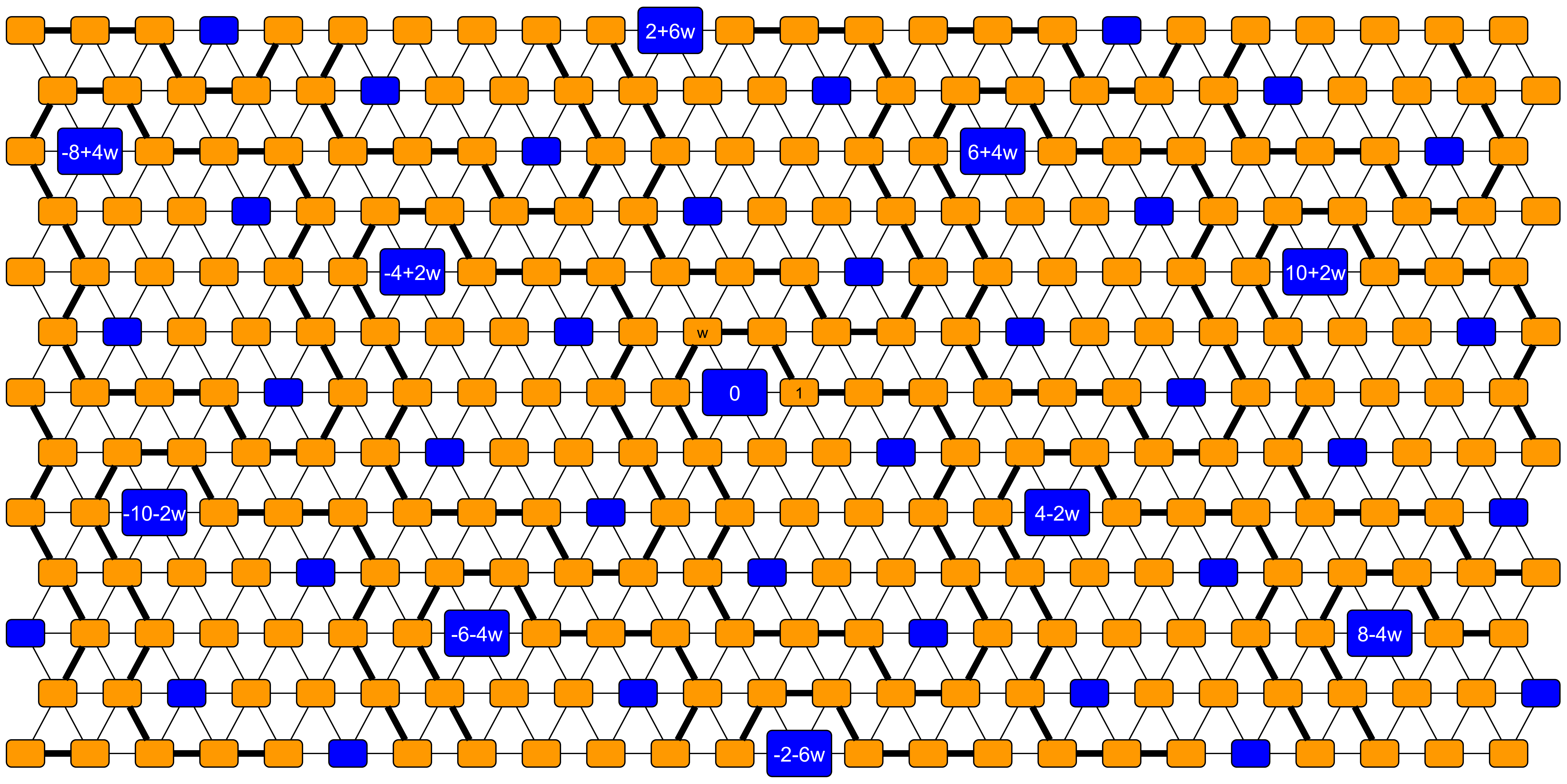}
		\caption{Subgroup $T_1$}
		\label{T1}
	\end{figure}
\end{landscape}

\begin{landscape}
	\begin{figure}[htbp]
		\includegraphics[scale=0.15]{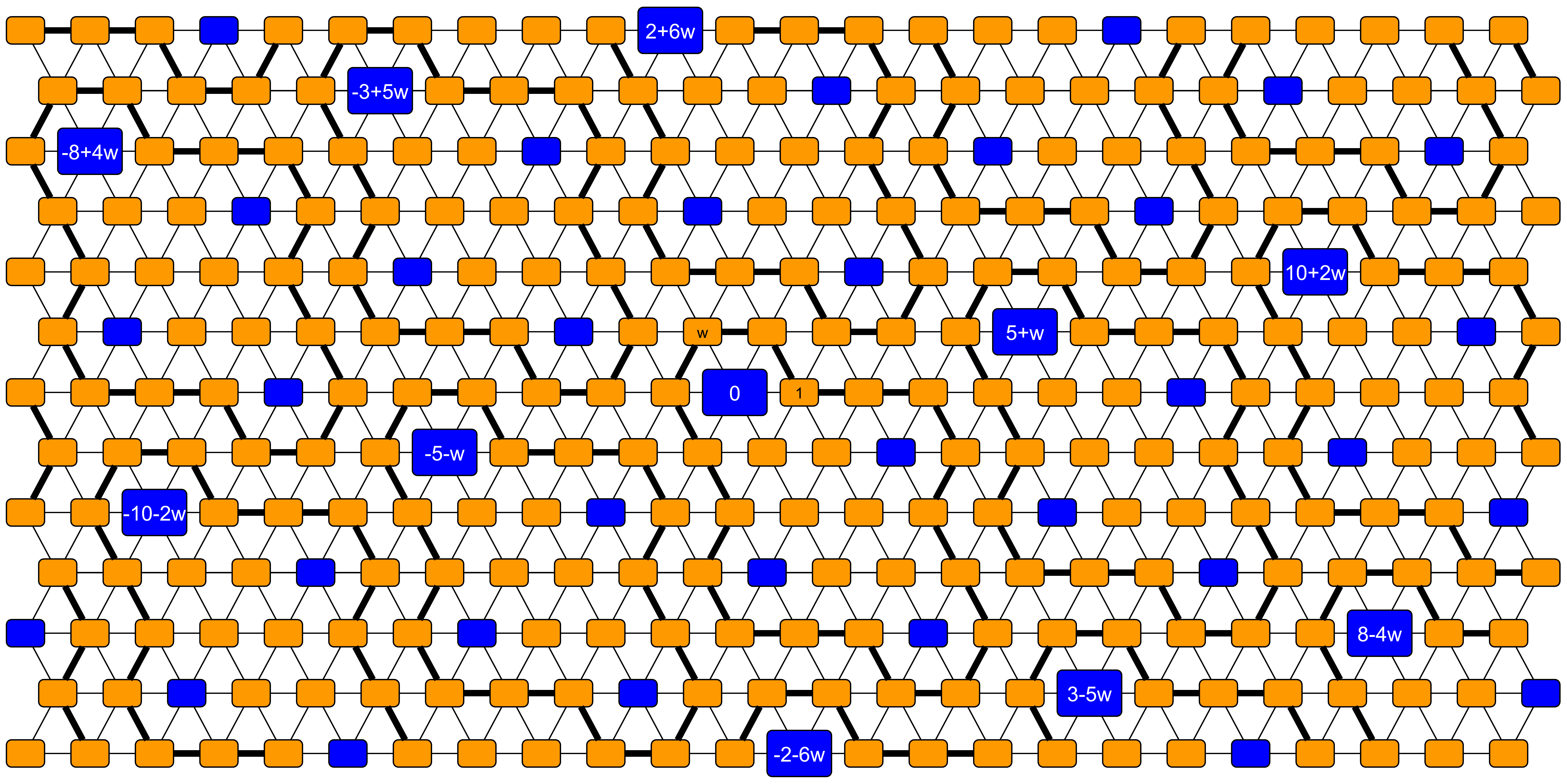}
		\caption{Subgroup $T_2$}
		\label{T2}
	\end{figure}
\end{landscape}

\begin{landscape}
	\begin{figure}[htbp]
		\includegraphics[scale=0.15]{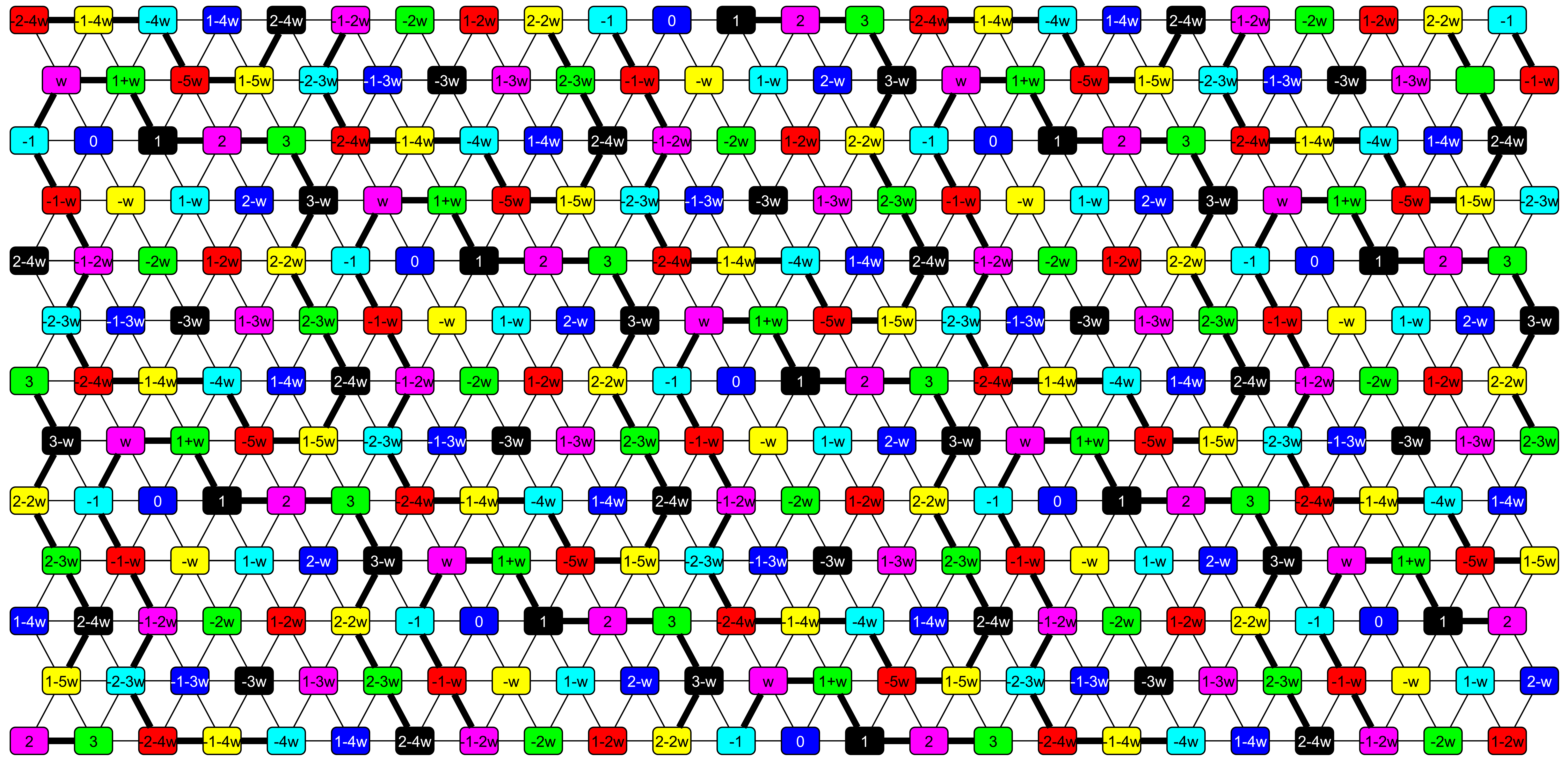}
		\caption{Quotient graph $\Delta_1$ and its partition into seven perfect 1-codes}
		\label{Delta1}
	\end{figure}
\end{landscape}

\begin{landscape}
	\begin{figure}[htbp]
		\includegraphics[scale=0.15]{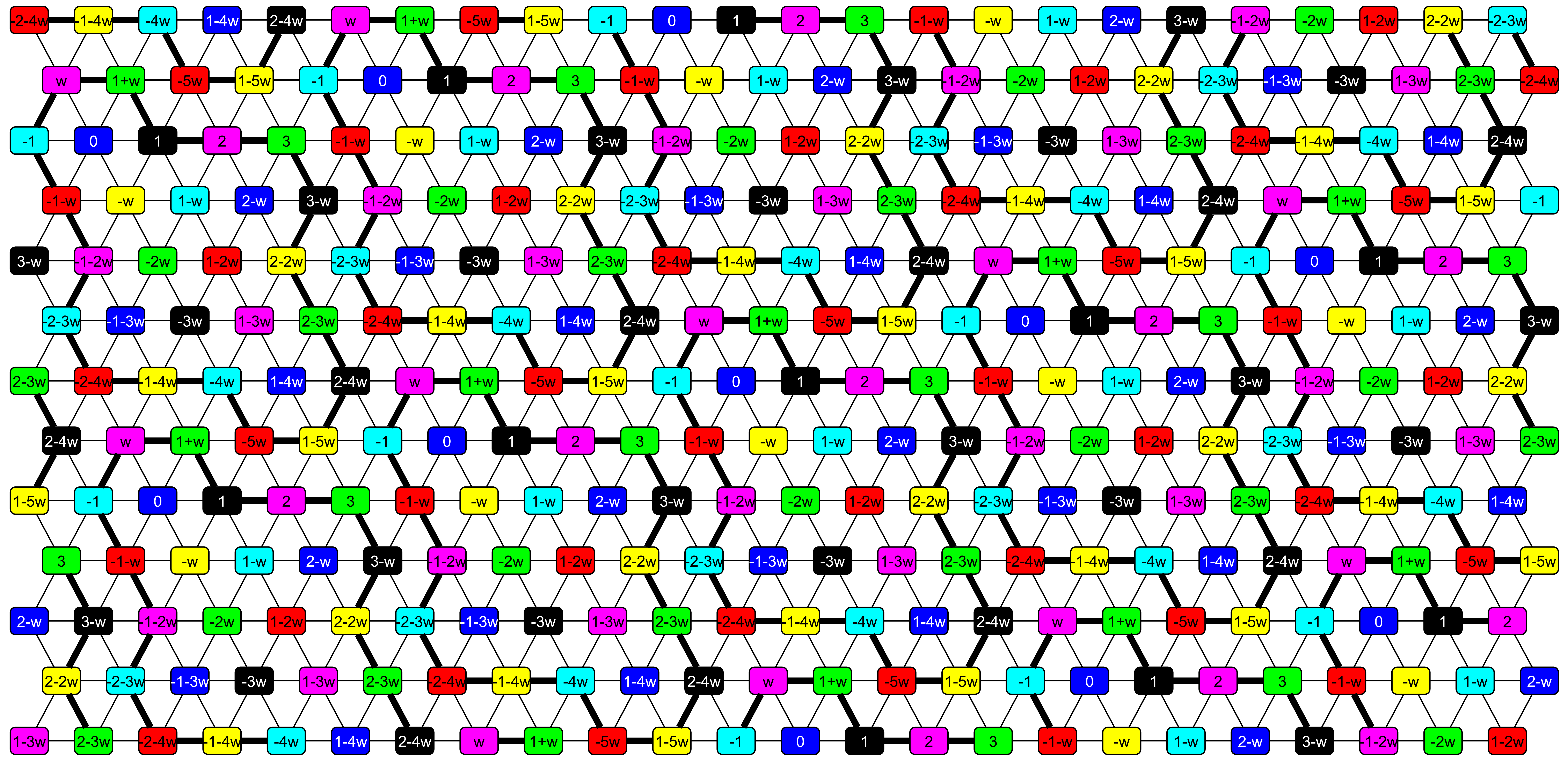}
		\caption{Quotient graph $\Delta_2$ and its partition into seven perfect 1-codes}
		\label{Delta2}
	\end{figure}
\end{landscape}


\begin{thebibliography}{00}

\bibitem{ADDK21}
A. Abiad, B. De Bruyn, J. D'haeseleer, J. H. Koolen, \emph{Neumaier graphs with few eigenvalues}, Designs, Codes and Cryptography (2021).\\ \url{https://doi.org/10.1007/s10623-021-00856-w}

\bibitem{ACDKZ23} A. Abiad, W. Castryck, M. De Boeck, J. H. Koolen, S. Zeijlemaker, \emph{An infinite class of Neumaier graphs and non-existence results}, Journal of Combinatorial Theory, Series A
Volume 193, January 2023, 105684.
\url{https://doi.org/10.1016/j.jcta.2022.105684}

\bibitem{BHK07}
S. Bang, A. Hiraki, J. H. Koolen, \emph{Delsarte clique graphs}, European Journal of Combinatorics, 28, 501--516 (2007).
\url{https://doi.org/10.1016/j.ejc.2005.04.015}

\bibitem{B73}
N. Biggs, \emph{Perfect codes in graphs}, J. Combin. Theory Ser. B 15, 288--296 (1973).
\url{https://doi.org/10.1016/0095-8956(73)90042-7}

\bibitem{BH12}
A.~E.~Brouwer, W.~H.~Haemers, \emph{Spectra of Graphs}, Springer, New York, 2012.

\bibitem{BCN89}
A. E. Brouwer, A. M. Cohen, and A. Neumaier, \emph{Distance-Regular Graphs}, Springer-Verlag, Berlin (1989).

\bibitem{EFHHH99}
M. Erickson, S. Fernando, W.H. Haemers, D. Hardy and J. Hemmeter, \emph{Deza graphs:
A generalization of strongly regular graphs}, Journal of Combinatorial Designs, 7, no. 6, 359--405 (1999).

\bibitem{E20} R. J. Evans, \emph{On regular induced subgraphs of edge-regular graphs}, PhD thesis, Queen Mary
University of London, 2020.

\bibitem{EGP19}
R. J. Evans, S. Goryainov, D. Panasenko, \emph{The smallest strictly Neumaier graph and its generalisations},
Electronic Journal of Combinatorics 26(2), \#2.29 (2019).
\url{https://doi.org/10.37236/8189}

\bibitem{GR01}
C.~Godsil, G.~Royle, Algebraic Graph Theory, Springer-Verlag, New York (2001).

\bibitem{GS14}
S. V. Goryainov, L. V. Shalaginov, \emph{Cayley-Deza graphs with fewer than 60 vertices}, Siberian Electronic Mathematical Reports, 11, 268--310 (2014). (in Russian)

\bibitem{GK18}
G. R. W. Greaves, J. H. Koolen, \emph{Edge-regular graphs with regular cliques}, European
Journal of Combinatorics, 71, 194--201 (2018).
\url{https://doi.org/10.1016/j.ejc.2018.04.004}

\bibitem{GK19}
G. R. W. Greaves, J. H. Koolen, \emph{Another construction of edge-regular graphs with regular cliques}, Discrete Mathematics, Volume 342, Issue 10,  2818--2820 (2019).
\url{https://doi.org/10.1016/j.disc.2018.09.032}

\bibitem{IM19} F. Ihringer, A. Munemasa,  
\emph{New strongly regular graphs from finite geometries via switching},
Linear Algebra and its Applications
580, 464--474 (2019).
\url{https://doi.org/10.1016/j.laa.2019.07.014}

\bibitem{K86}
J. Kratochv\'il, \emph{Perfect codes over graphs}, J. Combin. Theory Ser. B 40, 224--228 (1986).
\url{https://doi.org/10.1016/0095-8956(86)90079-1}

\bibitem{N81}
A. Neumaier, \emph{Regular cliques in graphs and special $1\frac{1}{2}$-designs}, Finite Geometries
and Designs, London Mathematical Society Lecture Note Series, 245--259 (1981).
\url{https://doi.org/10.1017/CBO9781107325579.027}

\bibitem{N92}
A. Neumaier, \emph{Completely regular codes}, \emph{Discrete Mathematics}, 106/107, 353--360 (1992).
\url{https://doi.org/10.1016/0012-365X(92)90565-W}

\bibitem{P22}
D. Panasenko, \emph{Database of small strictly Cayley-Deza graphs}, \url{http://alg.imm.uran.ru/dezagraphs/deza_cayleytab.html}

\bibitem{S15}
L. H. Soicher, \emph{On cliques in edge-regular graphs}, \emph{Journal of Algebra}, 421, 260--267 (2015).
\url{https://doi.org/10.1016/j.jalgebra.2014.08.028}

\bibitem{WQH19}
W. Wang, L. Qiu, Y. Hu,
\emph{Cospectral graphs, GM-switching and regular rational orthogonal matrices of level $p$},
Linear Algebra and its Applications, 563, 154--177 (2019).
\url{https://doi.org/10.1016/j.laa.2018.10.027}


\end{thebibliography}
\end{document}